\documentclass[12pt]{asl}

\usepackage{xspace}
\usepackage{dsfont}
\usepackage{ABT}

\newcommand{\Dzamonja}{D\v{z}amonja\xspace}
\newcommand{\NUF}{\textrm{NUF}}
\DeclareMathOperator{\SDom}{SDom}

\newtheorem{lemma}{Lemma}
\newtheorem*{mclaim}{Main Claim}

\title{Small $\mathfrak{u}_\kappa$ and large $2^\ka$ for supercompact $\kappa$}

\author{Andrew D. Brooke-Taylor}
\revauthor{Brooke-Taylor, Andrew D.}

\address{Group of Logic, Statistics and Informatics,\\
Graduate School of System Informatics,\\
Kobe University\\
Rokko-dai 1-1,\\
Nada, Kobe, 657-8501\\
Japan}

\email{andrewbt@kurt.scitec.kobe-u.ac.jp}

\thanks{Written while holding a
JSPS Postdoctoral Fellowship for Foreign Researchers at Kobe University and
supported by JSPS Grant-in-Aid no. 23 01765.}

\begin{document}

\begin{abstract}
Garti and Shelah~\cite{GaS:PCCI} state that one can force $\fru_\ka$
to be $\ka^+$ for supercompact $\ka$ with $2^\ka$ arbitrarily large,
using the technique of \Dzamonja and Shelah~\cite{DzS:UGSSC}.  
Here we spell out how this can be done.
\end{abstract}

\maketitle

\section{Introduction}

For any regular cardinal $\la$, we let
\[
\fru_\la=\min\{|\calB|: \calB
\text{ is a filter base for a uniform ultrafilter on }\la\}
\]
(recall that an ultrafilter is \emph{uniform} if every set in it has
the same cardinality).
A simple diagonalisation argument shows that $\fru_\la$ must be at least 
$\la^+$.
In \cite[Claim~2.2 (b)]{GaS:PCCI}, Garti and Shelah state that for
$\ka$ a supercompact cardinal, one can
force $\fru_\ka=\ka^+$ with $2^\ka$ arbitrarily large.  
They provide a short proof sketch, appealing to the arguments of 
\cite{DzS:UGSSC}.  We give here a detailed proof, based on the pair of
talks the author gave in the Kobe University set theory seminar on the topic,
closely following \cite{DzS:UGSSC}.
It should be noted that we have not discussed this with Garti or Shelah,
so what is presented might not exactly match their original intention,
but it seems (to the author) to be the most natural way to proceed.

We base our notation on that of
\Dzamonja and Shelah~\cite{DzS:UGSSC},
but do change much of it.
A particularly important change to note is that 
we use $p\leq q$ to mean that $p$ is a stronger condition than $q$,
in contrast with the usage in \cite{DzS:UGSSC}. 

The intention is that this note should be readable with no prior knowledge of
\cite{DzS:UGSSC} or \cite{GaS:PCCI}.  

\section{The partial order}

Let $\ka$ be a supercompact cardinal, and 
take $\Upsilon\geq 2^\ka$ such that $\Upsilon^{\ka}=\Upsilon$.
We will exhibit a forcing that makes $\fru_\ka=\ka^+$
and $2^{\ka}=\Upsilon$.  
To this end, we shall actually describe a forcing iteration of length
$\Upsilon^+$, which can be truncated at an appropriate point to obtain the
desired forcing (Garti and Shelah~\cite{GaS:PCCI} mention an 
iteration of length $\ka^+$; with \Dzamonja and Shelah's loose use of the
word ``iteration'' in \cite{DzS:UGSSC}, this matches the cofinality 
$\ka^+$ iteration we present).

We use the natural generalisation of Mathias forcing at measurable
$\ka$ rather than $\omega$ 
(or alternatively put, the natural generalisation of Prikry forcing 
to obtain $\ka$ sequences rather than those of length $\omega$;
we shall only ever be concerned with ultrafilters).
That is, for $\calD$ an ultrafilter on $\ka$, 
conditions in $\bbM^\ka_\calD$ are pairs
$(s,X)$ such that $s\in[\ka]^{<\ka}$ and $X\in \calD$, 
and $(t,Y)\leq(s,X)$ if and
only if $t$ end-extends $s$ and $(t\smallsetminus s)\cup Y\subseteq X$.

We define an iteration
$\langle P_i, \dot Q_i:i<\Upsilon^+\rangle$ as follows.
Let $G_i$ be $P_i$-generic; we describe $Q_i$ in $V[G_i]$.
Let NUF denote the set of normal ultrafilters on $\ka$ 
(in the measurable sense --- we only need supercompactness of $\ka$ to 
give us a Laver diamond --- see below).
The partial order $Q_i$ is then the sum over $\calD\in\NUF$ 
(interpreted in $V[G_i]$)
of the partial orders $\bbM^{\ka}_\calD$.  
That is, we take a maximum element $\bbone_{Q_i}$ (\Dzamonja and Shelah use 
$\emptyset$), and set
\[
Q_i=\{\bbone_{Q_i}\}\cup
\bigcup\left\{\{\calD\}\times\bbM^\ka_\calD:\calD\in\NUF\right\},
\]
with $p\leq q$ if and only if either
\begin{enumerate}
\item $q=\bbone_{Q_i}$, or
\item there are $\calD\in\NUF$ and $p_1\leq q_1\in\bbM^\ka_\calD$ such that
$p=(\calD,p_1)$ and $q=(\calD,q_1)$.
\end{enumerate}
We shall write $\bbone_\calD$ for
$(\calD,(\emptyset,\ka))$, the maximum element of the $\calD$ part of $Q_i$.

Now to the support of elements of $P_i$, $i\leq\Upsilon^+$.
We define the \emph{essential support of $p$}, $\SDom(p)$, by
\begin{multline*}
\SDom(p)=\big\{j\in\dom(p):\\
\lnot\left(p\restr j\forces_{P_j}p(j)\in\{\bbone_{Q_j}\}\cup
\{\bbone_\calD:\calD\in\NUF\}\right)\big\}.
\end{multline*}
Thus, $\SDom(p)$ is the set of coordinates at which $p$ does something
more than just choosing the ultrafilter for forcing at that stage.
We require that conditions in $P_{\Upsilon^+}$ 
have support bounded below $\Upsilon^+$ and
essential support of cardinality strictly less than $\ka$.
We freely identify $P_{\Upsilon^+}$ with $\bigcup_{i<\Upsilon^+}P_i$.

We call a condition $p\in P_i$ \emph{purely full in $P_i$}
or \emph{purely full in its domain} if for all $j<i$ we have
\[
p\restr j\forces_{P_j}p(j)\in\{\bbone_\calD:\calD\in\NUF\}.
\]
For $p\in P_i$ we write $P_i\downarrow p$ for 
$\{q\in P_i:q\leq p\}$; we will particularly be interested in the case
when $p$ is purely full in $P_i$.

\begin{lemma}[Claim 1.13 of \cite{DzS:UGSSC}]\label{kadirdcl}
$P_{\Upsilon^+}$ is $\ka$-directed-closed.
\end{lemma}
\begin{proof}
Each $\bbM^\ka_\calD$ is $\ka$-directed-closed, so this is standard.
\end{proof}

\begin{lemma}[Claim 1.16 of \cite{DzS:UGSSC}]\label{kacc}
Let $\tau$ be a $P_{\Upsilon^+}$-name and suppose that $p$ purely full in $P_i$
forces that $\tau$ names a set in the ground model $V$.
Then there is a $q\leq p$ purely full in its domain and a 
$(P_{\dom(q)}\downarrow q)$-name $\sigma$ such that
$q\forces\tau=\sigma$.
\end{lemma}
\begin{proof}
This is essentially just the $\ka^+$ chain condition.
Suppose the Lemma fails for $p$ purely full in its domain and 
$\tau$ a $P_{\Upsilon^+}$-name.
By recursion on $\zeta<\Upsilon^+$ we define 
$i_\zeta\in\Upsilon^+$, 
$\sigma_\zeta\in V^{P_{i_{\zeta}}}$, $p_\zeta$ purely full
in $P_{i_\zeta}$, $r_\zeta\in P_{i_{\zeta+1}}\downarrow p_{\zeta+1}$, 
and $x_\zeta\in V$,
such that:
\begin{enumerate}
\item $\langle i_\zeta\st\zeta<\Upsilon^+\rangle$ 
is strictly increasing continuous, 
\item $\langle p_\zeta\st\zeta<\Upsilon^+\rangle$ is decreasing, with 
$\dom(p_\zeta)=i_\zeta$ and $p_0=p$,
\item\label{rdecides} 
$r_\zeta\forces\tau=\check{x}_\zeta$, and $r_\zeta\perp r_\xi$ for
all $\xi<\zeta$, 
\item\label{sigmadef} $\sigma_\zeta=
\{\langle\check{w},r_\xi\rangle: \xi<\zeta\text{ and }w\in x_\xi\}$.
\end{enumerate}
Suppose we have $p_\xi$, $i_\xi$ and $\sigma_\xi$ for all $\xi\leq\zeta$,
and $r_\xi$ and $x_\xi$ for $\xi<\zeta$, satisfying 1--4.
From our assumption that the Lemma fails we have that 
\mbox{$p_\zeta\nforces\tau=\sigma_\zeta$}.
Using (\ref{rdecides}) and (\ref{sigmadef}), 
this means that $\{r_\xi:\xi<\zeta\}$ is not predense below $p_\zeta$.
So there is some $r_\zeta\leq p_\zeta$ in $P_{\Upsilon^+}$ 
incompatible with each
$r_\xi$, $\xi<\zeta$.  By extending if necessary, we may arrange that there
is some specific $x_\zeta\in V$ such that $r_\zeta\forces\tau=\check{x}_\zeta$,
and that $\dom(r_\zeta)=\sup(\dom(r_\zeta))$.
We may then define $p_{\zeta+1}$, $i_{\zeta+1}$ and $\sigma_{\zeta+1}$ 
from $r_\zeta$.  
Since continuity determines the values of $i_\zeta$, $p_\zeta$ and
$\sigma_\zeta$ for $\zeta$ a limit ordinal, 
this completes the recursive definition.

But now $\{r_\zeta:\zeta<\Upsilon^+\}$ is an antichain lying in
$\bigcup_{\zeta<\Upsilon^+}P_{i_\zeta}\downarrow p_\zeta$.  
This suborder of 
$P_{\Upsilon^+}$ is essentially the same as the $<\ka$-support iteration with
$\al$-th iterand $\bbM^\ka_{p^*(\al)}$ for every $\al<\Upsilon^+$, 
where $p^*=\bigcup_{\zeta<\Upsilon^+}p_\zeta$.
\Dzamonja and Shelah formalise this, and observe that this latter iteration
is $\ka^{+}$-cc, but perhaps it is easiest here to simply observe that the same
proof (a $\Delta$-system argument on  essential supports) shows that 
$\bigcup_{\zeta<\Upsilon^+}P_{i_\zeta}\downarrow p_\zeta$ is also $\ka^{+}$-cc.
In any case, we have a contradiction.
\end{proof}

\begin{lemma}\label{sizeUpsilon}
For $\Upsilon\leq j<\Upsilon^+$
and $p$ purely full in $P_j$,
$P_j\downarrow p$ has 
a dense suborder of
cardinality $\Upsilon$.
\end{lemma}
\begin{proof} 
This is again a use of the $\ka^+$-cc, along with the fact that
$\Upsilon^\ka=\Upsilon$.
We argue by induction.
For any $i<j$, a condition in $Q_i$ below $p(i)$ 
consists of a sequence from
$\ka$ of length less than $\ka$ and a subset of $\ka$, 
all of which may be determined by
$\ka$ many antichains from $P_i$. 
At limit stages, the result follows from the
fact that we are using $<\ka$ (essential) support.
\end{proof}

\section{Isolating an appropriate suborder}

Having defined $P_{\Upsilon^+}$ and observed some basic properties,
we now move to the key task of isolating a suborder that will
be what we actually force with to get $\fru_\ka<2^\ka$.
This suborder will be of the form $P_{\al}\downarrow p$ for
some condition $p$ purely full in $P_\al$;
the task thus boils down to constructing an appropriate $p$.

Since $P_{\Upsilon^+}$ is $\ka$-directed-closed (Lemma~\ref{kadirdcl}), 
it is natural to first apply a Laver
preparation \cite{Lav:prep} to ensure that $\ka$ remains supercompact after
our forcing.  For our argument, we will actually use it to obtain much more.
So let $h:\ka\to V_\ka$ be a Laver diamond, and let 
$\langle S_\al,\dot R_\be:\al\leq\ka,\be<\ka\rangle$
denote the Laver preparation defined using $h$ \cite{Lav:prep}.
That is, $S_\ka$ is a reverse Easton iteration, and the sequence 
$\langle\dot R_\be:\be<\ka\rangle$ and an auxiliary sequence of ordinals
$\langle\la_\be:\be<\ka\rangle$ are defined recursively 
according the the Laver diamond:
if $\be>\la_\ga$ for all $\ga<\be$, and $h(\be)$ is an ordered pair with
first term a $P_\be$-name for a
$\be$-directed-closed partial order and second term an ordinal,
then we set $(\dot R_\be,\la_\be)=h(\be)$; otherwise, we take $R_\be$ to be
the trivial forcing and $\la_\be$ to be $0$.

Take $\la\geq$\mbox{$|S_\ka*\dot P_{\Upsilon^+}|$} 
(this is probably overkill, but it makes no difference), and 
let $j:V\to M$ with ${}^{\la}M\subseteq M$ be a $\la$-supercompactness
embedding with critical point $\ka$ sent to $j(\ka)>\la$, such that 
$j(h)(\ka)=(P_{\Upsilon^+},\la)$.  
In particular, applying $j$ to the Laver preparation
$S_\ka$ we get $j(S_\ka)=S_{j(\ka)}^M={S_\ka*\dot P_{\Upsilon^+}*\dot S^*}^M$
for the appropriate tail iteration $S^*$ (in $M$).  Let us denote
$j(P_{\Upsilon^+})$ by $P'_{j(\Upsilon^+)}$.  Thus, applying $j$ to 
$S_\ka*\dot P_{\Upsilon^+}$ yields 
\[
j(S_\ka*\dot P_{\Upsilon^+})={S_\ka*\dot P_{\Upsilon^+}*\dot S^**\dot P'_{j(\Upsilon^+)}}^M.
\]
In the definition of the Laver preparation, if we have
a non-trivial iterand $\dot R_\al$ coming from 
$h(\al)=(\dot R_\al,\la_\al)$, then the 
subsequent iterands used are trivial until at least stage $\la_\al+1$, and
thereafter must be at least $|\la_\al|$-directed-closed.  
Since direct limits are only taken at inaccessible stages,
it follows that the tail
of the iteration from stage $\al+1$ onward is at least 
$|\la_\al|$-directed closed.  
In particular, we have in $M$ that
$\dot S^*$ is forced to be at least $\la$-directed-closed.
By elementarity, we also have that $P'_{j(\Upsilon^+)}$ is
$j(\ka)>\la$-directed-closed.

\begin{mclaim}[1.18 of \cite{DzS:UGSSC}]
In $V^{S_\ka}$, 
there exist sequences
\begin{align*}
\bar\al&=\langle\al_i:i<\Upsilon^+\rangle,\\
\bar p^*&=\langle p^*_i:i<\Upsilon^+\rangle,\text{ and}\\
\bar q^*&=\langle q^*_i=({}^1q_i,{}^2q_i):i<\Upsilon^+\rangle, 
\end{align*}
such that the following hold.
\end{mclaim}
\renewcommand{\theenumi}{\alph{enumi}}
\renewcommand{\theenumii}{\arabic{enumii}}
\begin{enumerate}
\item \label{als}
$\bar\al$ is a strictly increasing continuous
sequence of ordinals less than $\Upsilon^+$.
\item
Each $p^*_i$ is purely full in $P_{\al_i+1}$.
\item $\bar p^*$ is a decreasing sequence of conditions in
$P_{\Upsilon^+}$.
\item 
$\bar q^*\in M^{S_\ka}$,
and in $M^{S_\ka}$ we have for each $i<\Upsilon^+$ 
that
\[
(p^*_i,{}^1q_i)\in P_{\Upsilon^+}*\dot S^*
\]
and
\[
(p^*_i,{}^1q_i,{}^2q_i)\in P_{\Upsilon^+}*\dot S^**\dot P'_{j(\al_i+1)}.
\]
\item \label{pqqdecr}
In $M^{S_\ka}$, 
$\langle(p^*_i,{}^1q_i,{}^2q_i):i<\Upsilon^+\rangle$ is a decreasing 
sequence of conditions in
$P_{\Upsilon^+}*\dot S^**\dot P'_{\sup_{i<\Upsilon^+}(j(\al_i+1))}$.
\item \label{mastercond}
In $M^{S_\ka}$, $(p^*_{i+1},{}^1q_{i+1})$ forces that ${}^2q_{i+1}$ is a
common extension of
\[
\{j(r):r\in G_{P_{\al_i+1}}\}
\]
\item \label{Bs}
If $\dot B$ is an $S_\ka$-name for a $P_{\al_i+1}$-name for a subset
of $\ka$ then there is an $S_\ka*\dot P_{\Upsilon^+}$-name $\tau_{\dot B}$ 
for an element of
$\{0,1\}$ such that:
\begin{enumerate}
\item in $V$, $(\bbone_{S_\ka},\dot p^*_{i+1})$ forces $\tau_{\dot B}$ to be a 
$P_{\al_{i+1}+1}\downarrow p^*_{i+1}$-name, and
\item \label{kainjB}
$M\sat[(\bbone_{S_\ka},\dot p^*_{i+1},q^*_{i+1})\forces
\check\ka\in j(\dot B) \iff \tau_{\dot B}=\check1].
$
\end{enumerate}
\stepcounter{enumi}
\item \label{cfikadef}
If $\cf(i)>\ka$, then in $V^{S_\ka*\dot P_{\al_i}}$ we have that\\
$p^*_i(\al_i)=\Big\{\dot B[G_{P_{\al_i}}]:$
\parbox[c]{200pt}{
$\dot{B}$ is a $P_{\al_i}\downarrow(p^*_i\restr\al_i)$-name for a 
subset of $\ka$ and $\tau_{\dot B}[G_{P_{\al_i}}]=1$
}$\Big\}.$
In particular, this is a normal ultrafilter on $\ka$.
\end{enumerate}
(We have omitted (h) from our labelling so that it corresponds to that in 
\cite{DzS:UGSSC}.)

The crucial idea here is buried in item (g.2) 
We have an elementary embedding with critical point $\kappa$, 
and we want a nice normal ultrafilter on $\ka$, so as ever we define
it by saying that $B\subseteq\ka$ is in the ultrafilter 
if and only if $\ka$ is in $j(B)$.
In this context ``$\ka$ is in $j(B)$'' must be reinterpreted as 
``$\ka$ is forced to be in $j(\dot B)$'',
but these statements can be decided by boundedly much of the forcing
$P_{\Upsilon^+}$, 
as demonstrated by appeal to the technical device of the names $\tau_{\dot B}$.
In typical fashion, a long enough 
iteration with bookkeeping to ensure that
every name for a subset of $\ka$ is dealt with ``catches up with itself''.
At such closure stages we have that the resulting ultrafilter is defined
purely in terms of the construction that came before, and in particular
does not require a generic for the forcing $S^**\dot P'_{j(\Upsilon^+)}$
for its definition.
Moreover, these particular ultrafilters 
cohere with (indeed, extend) one another, 
allowing us to describe an ultrafilter in the
final extension in terms of those that came before, and 
to arrange that $\fru_\ka=\ka^+<2^{\ka}$ (see Theorem~\ref{usmall} below).

\begin{proof}[Proof of Main Claim]
Whilst the statement of the Main Claim might at first seem onerous, 
the sequences
$\bar\al$, $\bar p^*$ and $\bar q^*$ can actually be obtained by a relatively 
natural recursive construction,
making used of the $\la$-directed-closure of $S^*$ and $P'_{j(\Upsilon^+)}$
noted above.  
Indeed (\ref{als})--(\ref{pqqdecr}) merely set out the form of the sequences,
and whilst there is something to check, (\ref{cfikadef}) is 
actually giving part of the definition
for us.  Thus, the key to the recursive construction is
ensuring that (\ref{mastercond}) and (\ref{Bs}) hold.
We could begin with $\al_0=0$,
$p_0^*$ an arbitrary purely full element of $P_1$
(that is, $p_0^*=\bbone_\calD$ for some arbitrary $\calD$ in 
$\NUF^{V^{S_\ka}}$), and 
$q_0^*=(\bbone_{S^*},\dot\bbone_{P'_{j(\Upsilon^+)}})$.
But it will be notationally convenient if every $\al_i$ has cardinality 
$\Upsilon$, so let us take
$\al_0=\Upsilon$, $p^*_0$ an arbitrary purely full element of $P_{\Upsilon+1}$,
and 
$q_0^*=(\bbone_{S^*},\dot\bbone_{P'_{j(\Upsilon^+)}})$.

\noindent\emph{Choice of $\al_{i+1}$, $p^*_{i+1}$ and $q^*_{i+1}$,
given $\al_i$ and $p^*_i$ in $V^{S_\ka}$.}
First, towards the satisfaction of (\ref{mastercond}), note that
since $M$ is closed under taking $\la$-tuples,
we have
\[
\dot X_i=\{\langle j(\check r),r\rangle:
r\in P_{\al_i+1}\downarrow p^*_i\}^{V^{S_\ka}}
\in M^{S_\ka}.
\]
Of course, for generics containing $p_i^*$,
this $\dot X_i$ names $j``G_{P_{\al_i+1}}$, and 
\begin{align*}
(p^*_i,\dot\bbone_{S^*})\forces_{P_{\al_i+1}}&
\dot X_i\subseteq\dot P'_{j(\al_i)+1}\downarrow j(\check p^*_i)\land\\
&\dot X_i\text{ is directed}\land |\dot X_i|\leq\check\Upsilon.
\end{align*}
Thus, the $\la$-directed-closure of $P'_{j(\al_i)+1}$ 
allows us to find a master condition extending every condition in $X_i$, 
giving us the means to satisfy (\ref{mastercond}).
We postpone the use of this, as we will need to interleave it with our
construction towards the satisfaction of (\ref{Bs}).

In $V^{S_\ka}$, we have that $P_{\al_i+1}\downarrow p_i^*$ is a $\ka^+$-cc
partial order of size $\Upsilon$ (see Lemma~\ref{sizeUpsilon}, so there are 
$(\Upsilon^{\ka})^\ka=\Upsilon$ nice $P_{\al_{i}+1}\downarrow p_i^*$ names
for subsets of $\ka$.
Enumerate them in order type $\Upsilon$ as
$\langle\dot B^{i+1}_\zeta:\zeta<\Upsilon\rangle$.
To choose $p^*_{i+1}$ and $q^*_{i+1}$, we perform a further recursive 
construction, defining 
\begin{align*}
\langle\al^{i+1}_\zeta:\zeta<\Upsilon\rangle&\text{ increasing continuous,}\\
\langle p^{i+1}_{\zeta}:\zeta<\Upsilon\rangle&
\text{ decreasing continuous}\\
&\text{ with each $p^{i+1}_\zeta$ purely full in }
P_{\al^{i+1}_\zeta},\\
\langle q^{i+1}_\zeta=({}^1q^{i+1}_\zeta,{}^2q^{i+1}_\zeta):\zeta<\Upsilon
\rangle&\text{ (forced to be) decreasing, and }\\
\langle\tau_{\dot B^{i+1}_\zeta}:\zeta<\Upsilon\rangle&
\text{ a sequence of $S_\ka*\dot P_{\Upsilon^+}$-names}\\
&\text{ for elements of $\{0,1\}$.}
\end{align*}
Notice in particular that, whilst $\dom(p^*_i)=\al_i+1$, 
$\dom(p^{i+1}_\zeta)=\al^{i+1}_\zeta$.
Naturally enough, 
we start this recursion with $p^{i+1}_0=p^*_i$ and $q^{i+1}_0=q^*_i$.

Given $p^{i+1}_\zeta$ and $q^{i+1}_\zeta$, we want to extend to 
$p^{i+1}_{\zeta+1}$ and $q^{i+1}_{\ze+1}$ in a way that ``deals with''
$\dot B^{i+1}_\zeta$.  We ask whether there exists a $q$ that forces 
$\ka$ into $j(\dot B^{i+1}_\zeta)$ and which acts as a master condition for 
what has come before (perhaps confusingly, the \emph{negation} of this query is
referred to as ``the $\zeta$ question'' in \cite{DzS:UGSSC}).
Let us make this precise.  

We work in $M[G_{S_\ka*\dot P_{\al^{i+1}_\zeta}}]$, for some generic
$G_{S_\ka*\dot P_{\al^{i+1}_\zeta}}\ni(\bbone_S,\dot p^{i+1}_\zeta)$.
The values of $q$ and $\tau'_{\dot B^{i+1}_\zeta}$ 
that we describe there can then be combined below corresponding conditions in 
$P_{\Upsilon^+}$ to get single $P_{\Upsilon^+}$-names in the usual way.

We let 
\[
X^{i+1}_\zeta=\{j(r):r\in 
G_{S_\ka*\dot P_{\al^{i+1}_\zeta}}\};
\]
as for $X_i$, this
will be in $M[G_{S_\ka*\dot P_{\al^{i+1}_\zeta}}]$.

In $M[G_{S_\ka*\dot P_{\al^{i+1}_\zeta}}]$,
we ask whether there is
a condition $q=({}^1q,{}^2q)$ in $S^**\dot P'_{j(\Upsilon^+)}$ such that
\begin{itemize}
\item[$(\al)$] $q\leq q^{i+1}_\ze$
(and hence by induction $q\leq q^{i+1}_\xi$ for all $\xi\leq\zeta$), and
\item[$(\be)$] 
\begin{align*}
{}^1q\forces_{S^*}& \forall r\in\dot X^{i+1}_{\ze}({}^2q\leq r)\land\\
&{}^2q\in \dot P'_{j(\al^{i+1}_\zeta)}\downarrow j(p^{i+1}_\zeta)\land\\
&{}^2q\forces\check\ka\in j(\dot B^{i+1}_\zeta).
\end{align*}
\end{itemize}
Of course, the first conjunct in ($\be$) is towards making 
(\ref{mastercond}) hold, and the second conjunct is also to this end,
ensuring that ${}^2q$ does not interfere with parts of $j``G$
that arise later.
The final conjunct is, obviously, towards the satisfaction of (\ref{Bs}).

\emph{Case 1.} Suppose there is no $q$ that satisfies both $(\al)$ and $(\be)$.
Then we define $\tau'_{\dot B^{i+1}_{\zeta}}$ to be $0$ 
(in $M[G_{S_\ka*\dot P_{\al^{i+1}_\ze}}]$; in $M$ this of course contributes to
the definition of a 
$S_\ka*\dot P_{\al^{i+1}_\ze}$-name).

We claim that it is possible to find $q$ satisfying all of ($\al$) and
($\be$) except for the final conjunct of ($\be$),
and take $q^{i+1}_{\zeta+1}=({}^1q^{i+1}_{\zeta+1},{}^2q^{i+1}_{\zeta+1})$
to be such a condition.
That is, we take $q^{i+1}_{\zeta+1}\leq q^{i+1}_\ze$ such that
\begin{equation}
\tag{$\be'$}
{}^1q^{i+1}_{\zeta+1}\forces
\forall r\in\dot X^{i+1}_\ze(
{}^2q^{i+1}_{\zeta+1}\leq
r)\land
{}^2q^{i+1}_{\zeta+1}\in 
\dot P'_{j(\al^{i+1}_\zeta)}\downarrow j(p^{i+1}_\zeta).
\end{equation}
An appropriate condition
can be found since, as in the case of $X_i$ above, 
$(p^{i+1}_0,\dot\bbone_{S^*})$ forces 
$\dot X^{i+1}_\zeta$ to be small and directed,
and by induction,
the constraint on the support of ${}^2q$ in ($\be$) and ($\be'$)
ensures that all
previous ${}^2q^{i+1}_\xi$ are compatible with everything in $X^{i+1}_\zeta$.
(\Dzamonja and Shelah mention that $\dot X_i$ is in fact forced to be 
$\ka$-directed --- indeed,
it is a simple exercise to show that this is true for the
generic of any $<\ka$-strategically closed forcing.  But of course the 
$\ka$ in ``$\ka$-directed-closed'' refers to 
an upper bound on the size of the set, not the level of directedness, and so directedness suffices for our purposes.)

\emph{Case 2.}  If there is a $q$ satisfying ($\al$) and ($\be$), then we 
take $q^{i+1}_{\zeta+1}$ to be such a $q$, 
and take $\tau'_{\dot B^{i+1}_\zeta}=1$.

Stepping back to $M^{S_\ka}$ now, we may reconstruct 
$P_{\Upsilon^+}\downarrow p^{i+1}_\zeta$-names $q^{i+1}_{\zeta+1}$ and 
$\tau'_{\dot B^{i+1}_\zeta}$.
Since $\tau'_{\dot B^{i+1}_\zeta}$ is a $P_{\Upsilon^+}$-name for an element of
the ground model, by Lemma~\ref{kacc} there is a purely full in its domain
$p^{i+1}_{\zeta+1}\leq p^{i+1}_\zeta$ with domain some $\al^{i+1}_{\zeta+1}$
such that 
$p^{i+1}_{\zeta+1}$ forces $\tau'_{\dot B^{i+1}_\zeta}$ to be equivalent to a
$P_{\al^{i+1}_{\zeta+1}}\downarrow p^{i+1}_{\zeta+1}$-name; let us 
take $\tau_{\dot B^{i+1}_\zeta}$ to be such a name.
This concludes the description of the choice of $p^{i+1}_{\ze+1}$,
$\al^{i+1}_{\ze+1}$, $q^{i+1}_{\ze+1}$, and $\tau_{\dot B^{i+1}_\ze}$.

For limit ordinals $\zeta<\Upsilon$, we take 
$\al^{i+1}_\ze=\sup_{\xi<\ze}(\al^{i+1}_\xi)$ 
and $p^{i+1}_\ze=\bigcup_{\xi<\ze}p^{i+1}_\xi$.
Once again using the fact that
$S^**\dot P'_{j(\Upsilon^+)}$ is $\Upsilon^+$-directed-closed,
we may take $q^{i+1}_{\ze}$ to be a lower bound for
$\{q^{i+1}_\xi:\xi<\ze\}$.  
This concludes the recursion defining the sequences
$\langle \al^{i+1}_\ze\rangle$,
$\langle p^{i+1}_\ze\rangle$, 
$\langle q^{i+1}_\ze\rangle$, and
$\langle \tau_{\dot B^{i+1}_\ze}\rangle$. 
We now define $\al_{i+1}=\sup_{\ze<\Upsilon}(\al^{i+1}_\ze)$,
take $p^*_{i+1}$ to be any purely full condition in $P_{\al_{i+1}+1}$
extending $\bigcup_{\ze<\Upsilon}p^{i+1}_\zeta$ (so it is only
$p^*_{i+1}(\al_{i+1})$ that is arbitrary), and take
$q^*_{i+1}\in S^**\dot P'_{j(\al_{i+1})}$ such that
\[
(\bbone_{S_\ka},p^*_{i+1})\forces\forall\zeta<\Upsilon
(q^*_{i+1}\leq q^{i+1}_\zeta).
\]
By construction,
the requirements of the Main Claim (most notably items
(\ref{mastercond}) and (\ref{Bs})) are satisfied by these choices.

It remains to consider the choice of $\al_i$, $p^*_i$, and $q^*_i$
for $i<\Upsilon^+$ a limit ordinal.
Clearly we must take $\al_i=\sup_{j<i}\al_j$.
Likewise we must take $p^*_i$ purely full extending 
$\bigcup_{j<i}p*_j$, only leaving open the question of 
$p^*_i(\al_i)$:
if $\cf(i)\geq\ka$, we take $p^*_i(\al_i)$ as given by item (\ref{cfikadef})
of the Main Claim, and otherwise we take $p^*_i$ arbitrary.
We similarly take $q^*_i$ to be (forced to be) an arbitrary common extension
in $S^**\dot P'_{j(\al_i)}$
of $q^*_j$ for every $j<i$, as well as 
of every element of $X_i$; yet again this is possible by the level of
(directed) closure of $S^**\dot P'_{j(\al_i)}$.
The key items (\ref{mastercond}) and (\ref{Bs}) of the Main Claim only deal
with successor stages, so all that remains to check is that item
(\ref{cfikadef}) indeed yields a normal ultrafilter on $\ka$
when $\cf(i)>\ka$.
As \Dzamonja and Shelah note, this is a fairly
straightforward incorporation of 
master conditions into the
usual normal-ultrafilter-from-an-embedding argument. 
We shall nevertheless spell it out further.

First note that the definition of $p^*_i(\al_i)$ makes sense: if 
$\dot B$ is a $P_{\al_i}\downarrow(p^*_i\restr\al_i)$-name for a subset of
$\ka$, it is (equivalent to) a $P_{\al_j}$-name for some $j<i$.  This 
follows from the $\ka^+$-cc of $P_{\al_i}\downarrow(p^*_i\restr\al_i)$
(noted in the proof of Lemma~\ref{kacc}) and the fact that $\cf(i)>\ka$.

Suppose $G_{S_{\ka}*\dot P_{\al_i}}$ is 
an $S_\ka*\dot P_{\al_i}\downarrow(\bbone_{S_\ka},p^*_i\restr\al_i)$ generic,
and that in $V[G_{S*\dot P_{\al_i}}]$, $A\in p^*_i(\al_i)$ and $B\supseteq A$;
we wish to show that $B$ is also in $p^*_i(\al_i)$.
Choose names $\dot A$ and $\dot B$ for $A$ and $B$ respectively, and let $j<i$
be such that both $\dot A$ and $\dot B$ are 
$P_{\al_j+1}\downarrow p^*_j$-names.
Suppose $p\in G_{S*\dot P_{\al_i}}$ forces
$A\in p^*_i(\al_i)$ and $B\supseteq A$, that is, 
\begin{equation}\label{tA1BA}
\tag{$\dagger$}p\forces\tau_{\dot A}=1\land\dot B\supseteq\dot A. 
\end{equation}
By extending if necessary we may assume that $p\leq(\bbone_{S_\ka},p^*_{j+1})$.
Thus by item (g.2) of the Main Claim, in $M$ we have that 
\begin{equation}\label{tA1tB1}
\tag{$\ddagger$}(p,q^*_{j+1})\forces(\check\ka\in j(\dot A)\iff\tau_{\dot A}=1)\land
(\check\ka\in j(\dot B)\iff\tau_{\dot B}=1).
\end{equation}
It should be clear how we proceed from here, but note in particular the 
following point:
to deduce from \mbox{$(p,q^*_{j+1})\forces\check\ka\in j(\dot A)$} and 
\mbox{$p\forces\dot B\supseteq\dot A$} that
$(p,q^*_{j+1})\forces\check\ka\in j(\dot B)$, we need that $j$ lifts
to an elementary embedding
between the relevant forcing extensions. 
This is precisely why we needed to extend to
master conditions at every step of the iteration.
To be explicit, $p\forces\dot B\supseteq\dot A$ implies by elementarity only
that
$j(p)\forces j(\dot B)\supseteq j(\dot A)$.  However,
item (\ref{mastercond}) of the Main Claim ensures that 
$(p,q^*_{j+1})\leq j(p)$, so combining this with
(\ref{tA1BA}) and (\ref{tA1tB1}) we can indeed conclude that
\[
(p,q^*_{j+1})\forces\tau_{\dot B}=1.
\]
But now $\tau_{\dot B}$ is a $P_{\Upsilon^+}$-name, so it must be that
$p\forces\tau_{\dot B}=1$, and so $B\in p^*_i(\al_i)$.

The rest of the process of
checking that $p^*_i(\al_i)$ is a normal ultrafilter on $\ka$
is very similar, using the fact that we have taken master conditions
to get 
\begin{align*}
\text{from }p\forces&\dot B=\check\ka\smallsetminus\dot A&
\text{to }(p,q^*_{j+1})\forces&
\check\ka\notin j(\dot A)\iff \check\ka\in j(\dot B),\\
\text{from }p\forces&\dot B=\bigcap_{\ga<\de}\dot A_\ga&
\text{to }(p,q^*_{j+1})\forces&
\forall\ga<\de(\check\ka\in j(\dot A_\ga))\implies
\check\ka\in j(\dot B),\\
\text{\& from }p\forces&\dot B=\DiagInt_{\ga<\ka}\dot A_\ga&
\text{to }(p,q^*_{j+1})\forces&
\forall\ga<\ka(\check\ka\in j(\dot A_\ga))\implies
\check\ka\in j(\dot B).
\end{align*}
So we indeed have a normal ultrafilter, 
and hence a valid definition for $p^*_i(\al_i)$ for $i$ a
limit ordinal of cofinality greater than $\ka$.
This completes the proof of the Main Claim.
\end{proof}

With the Main Claim in hand we can finally prove the following.

\begin{thm}\label{usmall}
Let $\ka$ be a supercompact cardinal, and let $\Upsilon\geq2^\ka$ be a 
cardinal satisfying $\Upsilon^{\ka}=\Upsilon$.
Then there is a forcing extension in which $\ka$ remains supercompact,
$\fru_\ka=\ka^+$, and $2^\ka=\Upsilon$.
\end{thm}
\begin{proof}
With $\langle\al_i:i<\Upsilon^+\rangle$ as in the Main Claim,
we take the forcing $S_\ka*\dot P_{\al_i}\downarrow(p^*_i\restr\al_i)$ for 
$i=\ka^+\cdot\ka^+$ (the ordinal square of $\ka^+$).
Let $G$ be $S_\ka*\dot P_{\al_i}\downarrow(p^*_i\restr\al_i)$-generic over
$V$.
Since we begin with the Laver preparation, $\ka$ 
certainly remains supercompact in $V[G]$, and since 
$|\al_{\ka^+\cdot\ka^+}|=\Upsilon$, $2^\ka=\Upsilon$ in $V[G]$.

To show that $\fru_\ka=\ka^+$ in the generic extension, we consider the
normal ultrafilter given by item (\ref{cfikadef}) of the Main Claim,
which would be $p^*_{\ka^+\cdot\ka^+}(\al_{\ka^+\cdot\ka^+})$ 
if we continued the iteration.
That is, in $V[G]$ we consider\\
\[
\calD=
\left\{B\subset\ka:\tau_{\dot B}[G]=1
\right\}
\]
(of course, whether $B$ is in $\calD$ is independent of the choice of $\dot B$
by the construction of the names $\tau_{\dot B}$).
Since conditions in $P_{\al_i}\downarrow(p^*_i\restr\al_i)$ have
essential support bounded below $\al_{\ka^+\cdot\ka^+}$, 
each subset $B$ of $\ka$ in the
extension is named by some stage of the iteration prior to
$\al_{\ka^+\cdot\ka^+}$.  In particular,
either $B$ or $\ka\smallsetminus B$
will appear in the ultrafilter 
$p^*_{\ka^+\cdot\de}(\al_{\ka^+\cdot\de})[G]$ for $\de<\ka^+$
sufficiently large,
and this ultrafilter
is determined by item (\ref{cfikadef}) of the Main Claim.
Hence, we have that
$B\in p^*_{\ka^+\cdot\de}(\al_{\ka^+\cdot\de})[G]$ if and only
if $\tau_{\dot B}[G]=1$, if and only if $B\in \calD$.
We thus have that 
\[
\calD=\bigcup_{\de<\ka^+}
p^*_{\ka^+\cdot\de}(\al_{\ka^+\cdot\de})[G].
\]
Now at stage $\al_{\ka^+\cdot\de}$ of the iteration, we are forcing with
$\bbM^\ka_{p^*_{\ka^+\cdot\de}}$, and so the generic subset of $\ka$
at this stage, $X_{\al_{\ka^+\cdot\de}}$, is almost below every element of
$p^*_{\ka^+\cdot\de}(\al_{\ka^+\cdot\de})[G]$.
Hence, $\calD$ is generated by the set
\[
\{Y\subseteq\ka:\exists\de<\ka^+(|Y\triangle X_{\al_{\ka^+\cdot\de}}|<\ka)\},
\]
which has cardinality $\ka^+$.
\end{proof}

\bibliographystyle{asl}
\bibliography{logic}

\end{document}